\documentclass[12pt]{amsart}

\usepackage{amsmath,amssymb,amscd,amsthm,graphicx,enumerate}

\usepackage{hyperref}
\hypersetup{breaklinks=true}

\numberwithin{equation}{section}
\theoremstyle{plain}

\makeatletter
\newtheorem*{rep@theorem}{\rep@title}
\newcommand{\newreptheorem}[2]{%
\newenvironment{rep#1}[1]{%
 \def\rep@title{#2 \ref{##1}}%
 \begin{rep@theorem}}%
 {\end{rep@theorem}}}
\makeatother

\newcommand{\Balpha}{\mbox{$\hspace{0.1em}\rule[0.01em]{0.05em}{0.39em}\hspace{-0.21em}\alpha$}}

\newtheorem{theorem}[equation]{Theorem}
\newreptheorem{theorem}{Theorem}
\newtheorem{thm}[equation]{Theorem}
\newtheorem{proposition}[equation]{Proposition}
\newtheorem{lemma}[equation]{Lemma}
\newtheorem{corollary}[equation]{Corollary}

\newtheorem{openproblem}[equation]{Open problem}
\newtheorem{claim}[equation]{Claim}
\newtheorem{exercise}[equation]{Exercise}

\theoremstyle{remark}
\newtheorem{remark}[equation]{Remark}

\theoremstyle{definition}
\newtheorem{definition}[equation]{Definition}
\newtheorem{defn}[equation]{Definition}

\newtheorem*{question*}{Question}

\newtheorem{example}[equation]{Example}
\newcommand{\RR}{\mathbb{R}}
\newcommand{\Grad}{\nabla}
\newcommand{\abs}[1]{\lvert#1\rvert}
\newcommand{\ip}[1]{\langle#1\rangle}
\newcommand{\norm}[1]{\lVert#1\rVert}
\newcommand{\zu}{Z_\ast}

\newcommand{\K}{{\mathcal K}}

\newcommand{\M}{{\mathcal M}}

\newcommand{\R}{\mathbb R}

\renewcommand{\S}{{\mathcal S}}

\newcommand{\dist}{\operatorname{dist}}

\newcommand{\Int}{\operatorname{Int}}

\newcommand{\al}{\alpha}

\newcommand{\D}{\partial}
\newcommand{\de}{\delta}

\newcommand{\Lap}{\Delta}

\newcommand{\eps}{\varepsilon}

\newcommand{\la}{\lambda}

\newcommand{\ol}{\overline}

\newcommand{\ra}{\rightarrow}

\providecommand{\abs}[1]{\lvert #1\rvert}
\providecommand{\norm}[1]{\lvert\lvert #1\rvert\rvert}

\def\XXint#1#2#3{{\setbox0=\hbox{$#1{#2#3}{\int}$}
     \vcenter{\hbox{$#2#3$}}\kern-.5\wd0}}

\begin{document}

\title[Mean curvature flow]{Lectures on mean curvature flow}
\author{Robert Haslhofer}

\thanks{I thank Bruce Kleiner for very fruitful collaboration, Hojoo Lee, Luigi Ambrosio, Carlo Mantegazza and Andrea Mennucci for organizing the summer schools at KIAS Seoul and SNS Pisa, and all the participants for their questions and feedback.}

\date{\today}
\maketitle


\begin{abstract}
A family of hypersurfaces evolves by mean curvature flow if the velocity at each point is given by the mean curvature vector.
Mean curvature flow is the most natural evolution equation in extrinsic geometry, and has been extensively studied ever since the pioneering work of Brakke \cite{brakke} and Huisken \cite{Huisken_convex}.
In the last 15 years, White developed a far-reaching regularity and structure theory for mean convex mean curvature flow \cite{white_size,white_nature,white_subsequent}, and Huisken-Sinestrari constructed a flow with surgery for two-convex hypersurfaces \cite{huisken-sinestrari1,huisken-sinestrari2,huisken-sinestrari3}.
In this course, I first give a general introduction to the mean curvature flow of hypersurfaces and then present joint work with Bruce Kleiner \cite{HK1,HK2}, where we give a streamlined and unified treatment of the theory of White and Huisken-Sinestrari.
These notes are from summer schools at KIAS Seoul and SNS Pisa.
\end{abstract}

\tableofcontents

\section{Introduction to the mean curvature flow}

In this first lecture, I will give a quick informal introduction to the mean curvature flow. For more comprehensive introductions, I recommend the books by Ecker \cite{Ecker_book} and Mantegazza \cite{Mantegazza_book}.

A smooth family of embedded hypersurfaces $\{M_t\subset \R^{n+1}\}_{t\in I}$ moves by mean curvature flow if
\begin{equation}\label{eq_mcf}
\partial_t x = \vec{H}(x)
\end{equation}
for $x\in M_t$ and $t\in I$. Here, $I\subset \R$ is an interval, $\partial_t x$ is the normal velocity at $x$, and $\vec{H}(x)$ is the mean curvature vector at $x$.

If we write $\vec{H}=H\vec{\nu}$ for a unit normal $\vec{\nu}$, then $H$ is given by the sum of the principal curvatures, $H=\lambda_1+\ldots \lambda_n$.

\begin{example}[Shrinking spheres and cylinders] If $M_t=\partial B^{n+1}_{r(t)}\subset\R^{n+1}$, then \eqref{eq_mcf} reduces to an ODE for the radius, namely
$\dot{r}=-n/r$.
The solution with $r(0)=R$ is $r(t)=\sqrt{R^2-2nt}$, $t\in (-\infty,R^2/2n)$.
Similarly, we have the shrinking cylinders $M_t=\R^j\times\partial B^{n+1-j}_{r(t)}\subset\R^{n+1}$ with $r(t)=\sqrt{R^2-2(n-j)t}$, $t\in (-\infty,R^2/2(n-j))$.
\end{example}

\begin{exercise}[Grim reaper]
Show that for $n=1$ an explicit solution is given by $M_t=\textrm{graph}(u_t)$, where $u_t(p)=t-\log \cos p$ with $p\in(-\tfrac{\pi}{2},\tfrac{\pi}{2})$.
\end{exercise}

Instead of viewing the mean curvature flow as an evolution equation for the hypersurfaces $M_t$, we can also view it as an evolution equation for a smooth family of embeddings $X=X(\cdot,t): M^n \times I \rightarrow \RR^{n+1}$ with $M_t=X(M,t)$.
Setting $x=X(p,t)$, equation \eqref{eq_mcf}  then takes the form
\begin{equation}\label{eq_mcf2}
\partial_t X(p,t) = \Lap_{M_t} X(p,t).
\end{equation}

The fundamental idea of geometric flows is to deform a given geometric object into a nicer one, by evolving it by a heat-type equation.
This indeed works very well, as illustrated by the following theorem.

\begin{theorem}[Huisken \cite{Huisken_convex}]
Let $M_0\subset \R^{n+1}$ be a closed embedded hypersurface. If $M_0$ is convex, then the mean curvature flow $\{M_t\}_{t\in[0,T)}$ starting at $M_0$ converges to a round point.
\end{theorem}

The convex case ($\lambda_1\geq 0,\ldots ,\lambda_n\geq 0$) is of course very special.
In more general situations, we encounter the formation of singularities.

\begin{example}[Neckpinch singularity]
If $M_0$ looks like a dumbbell, then the neck pinches off. As blowup limit we get a shrinking cylinder.
\end{example}

\begin{exercise}[Parabolic rescaling]
Let $\{M_t\subset \R^{n+1}\}$ be a mean curvature flow of hypersurfaces, and let $\lambda>0$.
Let  $\{M_{t'}^\lambda\}$ be the family of hypersurfaces obtained by the parabolic rescaling $x'=\lambda x$, $t'=\lambda^2 t$, i.e. let $M_{t'}^\lambda=\lambda M_{\lambda^{-2}t'}$. Show that $\{M_{t'}^\lambda\}$ again solves \eqref{eq_mcf}.
\end{exercise}

The formation of singularities is the most important topic in the study of geometric flows.
Roughly speaking, the main questions are: How do singularities and regions of high curvature look like? Can we continue the flow through singularities? If we continue the flow in a weak way, what is the size and the structure of the singular set? Can we flow through singularities in a controlled way by performing surgeries?

The key to answer the above questions is of course to prove strong enough estimates; this will be the main topic of the following lectures.

We end this first lecture, by summarizing a few basic properties of the mean curvature flow, see e.g. \cite{Ecker_book,Mantegazza_book} for more on that.

First, by standard parabolic theory, given any compact initial hypersurface $M_0\subset \R^{n+1}$ (say smooth and embedded), there exists a unique smooth solution $\{M_t\}_{t\in [0,T)}$ of \eqref{eq_mcf} starting at $M_0$, and defined on a maximal time interval $[0,T)$. The maximal time $T$ is characterized by the property that the curvature blows up, i.e.
$\lim_{t\to T}\max_{M_t}\abs{A}=\infty$.

Second, if $M_t$ and $N_t$ are two compact mean curvature flows, then $\dist(M_t,N_t)$ is nondecreasing in time. In particular, by comparison with spheres, the maximal time $T$ above is indeed finite.

Third, the evolution equation \eqref{eq_mcf} implies evolution equations for the induced metric $g_{ij}$, the area element $d\mu$, the normal vector $\vec{\nu}$, the mean curvature $H$, and the second fundamental form $A$.

\begin{proposition}[Evolution equations for geometric quantities]\label{prop_evol_eq}
If $\{M_t\subset \R^{n+1}\}$ evolves by mean curvature flow, then 
\begin{equation}\label{eq_evol}
\begin{array}{lll}
\partial_t g_{ij}=-2HA_{ij} & \partial_td\mu=-H^2 d\mu &\partial_t\vec{\nu}=-\nabla H\\
\partial_t H=\Lap H+\abs{A}^2 H & \partial_t A^i_j=\Lap A^i_j+\abs{A}^2 A^i_j.&
\end{array}
\end{equation}
\end{proposition}
For example, the evolution of $g_{ij}=\partial_i X\cdot \partial_j X$ is computed via
\begin{equation}
\partial_t g_{ij}=2\partial_i (H\vec{\nu})\cdot \partial_j X=2H\partial_i \vec{\nu}\cdot \partial_j X=-2HA_{ij}.
\end{equation}

\begin{exercise}[Evolution of the area element]
Show that if $G=G(t)$ is a smooth family of invertible matrices, then $\tfrac{d}{dt} \ln\det G=\textrm{tr}_G{\tfrac{d}{dt}G}$. Use this to derive the evolution equation for $d\mu=\sqrt{\det g_{ij}}d^nx$.
\end{exercise}

In particular, if $M_0$ is compact the total area decreases according to
\begin{equation}\label{eq_areamon}
\frac{d}{dt}\textrm{Area}(M_t)=-\int_{M_t} H^2 d\mu.
\end{equation}

Finally, using Proposition \ref{prop_evol_eq} and the maximum principle we obtain:

\begin{proposition}[Preserved curvature conditions]
Let $\{M_t\subset \R^{n+1}\}$ be a mean curvature flow of compact hypersurfaces.
If $H\geq 0$ at $t=0$, then $H\geq 0$ for all $t>0$.
Similarly, the conditions $\lambda_1+\ldots +\lambda_k \geq 0$, and $\lambda_1+\ldots +\lambda_k \geq \beta H$ are also preserved along the flow.
\end{proposition}

\section{Monotonicity formula and local regularity theorem}

In this second lecture, we discuss Huisken's monotonicity formula and the local regularity theorem for the mean curvature flow.

Recall that by equation \eqref{eq_areamon} the total area is monotone under mean curvature flow. However, since $\textrm{Area}(\lambda M)=\lambda^n \textrm{Area}(M)$, this is not that useful when considering blowup sequences with $\lambda\to \infty$. A great advance was made by Huisken, who discovered a scale invariant monotone quantity.
To describe this, let $\M=\{M_t\subset \R^{n+1}\}$ be a smooth mean curvature flow of hypersurfaces, say with at most polynomial volume growth,
let $X_0=(x_0,t_0)$ be a point in space-time, and let
\begin{equation}
 \rho_{X_0}(x,t)=(4\pi(t_0-t))^{-n/2} e^{-\frac{\abs{x-x_0}^2}{4(t_0-t)}}\qquad (t<t_0),
\end{equation}
be the $n$-dimensional backwards heat kernel centered at $X_0$.

\begin{theorem}[Huisken's monotonicity formula \cite{Huisken_monotonicity}]\label{thm_huisken_mon}
\begin{equation}\label{eq_huisken_mon}
 \frac{d}{dt}\int_{M_t} \rho_{X_0} d\mu = -\int_{M_t} \left|\vec{H}-\frac{(x-x_0)^\perp}{2(t-t_0)}\right|^2 \rho_{X_0} d\mu\qquad (t<t_0).
\end{equation}
\end{theorem} 

Huisken's monotonicity formula \eqref{eq_huisken_mon} can be thought of as weighted version of \eqref{eq_areamon}. A key property is its invariance under  rescaling.

\begin{exercise}[Parabolic rescaling]
Let $x'=\lambda(x-x_0)$, $t'=\lambda^2(t-t_0)$, and consider the rescaled flow $M^\lambda_{t'}=\lambda(M_{t_0+\lambda^{-2}t'}-x_0)$. Prove that
\begin{equation}
 \int_{M_t} \rho_{X_0}(x,t) \, d\mu_t(x) = \int_{M^\lambda_{t'}} \rho_{0}(x',t')\,  d\mu_{t'}(x')\qquad (t'<0).
\end{equation}
\end{exercise}

Moreover, the equality case of \eqref{eq_huisken_mon} exactly characterizes the selfsimilarly shrinking solutions (aka shrinking solitons).

\begin{exercise}[Shrinking solitons]
Let $\{M_t\subset \R^{n+1}\}_{t\in (-\infty,0)}$ be an ancient solution of the mean curvature flow. Prove that
$\vec{H}-\frac{x^\perp}{2t}=0$ for all $t<0$ if and only if $M_t=\sqrt{-t}M_{-1}$ for all $t<0$.
\end{exercise}

\begin{proof}[Proof of Theorem \ref{thm_huisken_mon}]
Wlog $X_0=(0,0)$. The proof essentially amounts to deriving belows pointwise identity \eqref{eq_pointwise} for $\rho=\rho_0$.

Since the tangential gradient of $\rho$ is given by $\nabla^{M_t}\rho = D\rho-(D\rho\cdot\vec{\nu}) \vec{\nu}$,
the intrinsic Laplacian of $\rho$ can be expressed as
\begin{equation}
\Lap_{M_t}\rho = \textrm{div}_{M_t}\nabla^{M_t}\rho
=\textrm{div}_{M_t}D\rho+\vec{H}\cdot D\rho.
\end{equation}
Observing also that $\tfrac{d}{dt}\rho=\partial_t \rho+\vec{H}\cdot D\rho$, we compute
\begin{align}
(\tfrac{d}{dt}+\Lap_{M_t})\rho&=\partial_t \rho+\textrm{div}_{M_t}D\rho+2\vec{H}\cdot D\rho\nonumber\\
&=\partial_t \rho+\textrm{div}_{M_t}D\rho+\frac{\abs{\nabla^\perp \rho}^2}{\rho}-\abs{\vec{H}-\frac{\nabla^\perp \rho}{\rho}}^2 \rho+H^2\rho.
\end{align}
We can now easily check that $\partial_t \rho+\textrm{div}_{M_t}D\rho+\frac{\abs{\nabla^\perp \rho}^2}{\rho}=0$. Thus 
\begin{equation}\label{eq_pointwise}
(\tfrac{d}{dt}+\Lap_{M_t}-H^2)\rho=-\abs{\vec{H}-\frac{x^\perp}{2t}}^2 \rho.
\end{equation}
Using also the evolution equation $\tfrac{d}{dt}d\mu =-H^2d\mu $, we conclude that
\begin{equation}
 \frac{d}{dt}\int_{M_t} \rho\, d\mu = -\int_{M_t} \left|\vec{H}-\frac{x^\perp}{2t}\right|^2 \rho\,  d\mu  \qquad (t<0).
\end{equation}
This proves the theorem.
\end{proof}

\begin{remark}[Local version \cite{Ecker_book}]
If $M_t$ is only defined locally, say in $B(x_0,\sqrt{4n}\rho)\times (t_0-\rho^2,t_0)$, then we can use the cutoff function
$\varphi^\rho_{X_0}(x,t)=(1-\tfrac{\abs{x-x_0}^2+2n(t-t_0)}{\rho^2})_+^3$.
Since $(\tfrac{d}{dt}-\Lap_{M_t})\varphi^\rho_{X_0}\leq 0$ we still get the monotonicity inequality
\begin{equation}\label{app_loc_mon}
 \frac{d}{dt}\int_{M_t} \rho_{X_0}\varphi^\rho_{X_0} d\mu \leq -\int_{M_t} \left|\vec{H}-\frac{(x-x_0)^\perp}{2(t-t_0)}\right|^2 \rho_{X_0}\varphi^\rho_{X_0} d\mu.
\end{equation}
\end{remark}

The monotone quantity appearing on the left hand side,
\begin{equation}
\Theta^\rho(\M,X_0,r)=\int_{M_{t_0-r^2}} \rho_{X_0}\varphi_{X_0}^\rho d\mu,
\end{equation}
is called the Gaussian density ratio. Note that $\Theta^\infty(\M,X_0,r)\equiv 1$ for all $r>0$ if and only if $\M$ is a multiplicity one plane containing $X_0$.

We will now discuss the local regularity theorem for the mean curvature flow, which gives definite curvature bounds in a neighborhood of definite size, provided the Gaussian density ratio is close to one.

Since time scales like distance squared, the natural neighborhoods to consider are parabolic balls $P(x_0,t_0,r)=B(x_0,r)\times (t_0-r^2,t_0]$.

\begin{theorem}[Local regularity theorem \cite{brakke,white_regularity}]\label{app_thm_easy_brakke}
There exist universal constants $\eps>0$ and $C<\infty$ with the following property.
If $\M$ is a smooth mean curvature flow in a parabolic ball $P(X_0,4n\rho)$ with
\begin{equation}
 \sup_{X\in P(X_0,r)}\Theta^{\rho}(\M,X,r)<1+\eps
\end{equation}
for some $r\in(0,\rho)$, then
\begin{equation}
 \sup_{P(X_0,r/2)}\abs{A}\leq {C}r^{-1}.
\end{equation}
\end{theorem}

\begin{remark}
If $\Theta<1+\frac{\eps}{2}$ holds at some point and some scale, then $\Theta<1+\eps$ holds at all nearby points and somewhat smaller scales.
\end{remark}

\begin{proof}[Proof of Theorem \ref{app_thm_easy_brakke}]
 Suppose the assertion fails. Then there exist a sequence of smooth flows $\M^j$ in $P(0,4n \rho_j)$, for some $\rho_j> 1$, with
\begin{equation}
 \sup_{X\in P(0,1)}\Theta^{\rho_j}(\M^j,X,1)<1+j^{-1},
\end{equation}
but such that there are points $X_j\in P(0,1/2)$ with $\abs{A}(X_j)> j$.

By point selection, we can find $Y_j\in P(0,3/4)$ with $Q_j=\abs{A}(Y_j)> j$ such that
\begin{equation}\label{app_brakke_point_sel}
 \sup_{P(Y_j,j/10Q_j)}\abs{A}\leq 2 Q_j.
\end{equation}
Let us explain how the point selection works: Fix $j$. If $Y^0_j=X_j$ already satisfies (\ref{app_brakke_point_sel}) with $Q^0_j=\abs{A}(Y^0_j)$, we are done. Otherwise, there is a point $Y^1_j\in P(Y_j^0,j/10Q^0_j)$ with $Q^1_j=\abs{A}(Y^1_j)>2Q^0_j$.
If $Y^1_j$ satisfies (\ref{app_brakke_point_sel}), we are done. Otherwise, there is a point $Y^2_j\in P(Y_j^1,j/10Q^1_j)$ with $Q^2_j=\abs{A}(Y^2_j)>2Q^1_j$, etc.
Note that $\frac{1}{2}+\frac{j}{10Q_j^0}(1+\frac{1}{2}+\frac{1}{4}+\ldots)<\frac{3}{4}$. By smoothness, the iteration terminates after a finite number of steps, and the last point of the iteration lies in $P(0,3/4)$ and satisfies (\ref{app_brakke_point_sel}).

Continuing the proof of the theorem, let $\hat\M^j$ be the flows obtained by shifting $Y_j$ to the origin and parabolically rescaling by $Q_j=\abs{A}(Y_j)\to\infty$.
Since the rescaled flow satisfies $\abs{A}(0)=1$ and $\sup_{P(0,j/10)}\abs{A}\leq 2$, we can pass smoothly to a nonflat global limit. On the other hand, by the rigidity case of (\ref{app_loc_mon}), and since
\begin{equation}
 \Theta^{\hat\rho_j}(\hat\M^j,0,Q_j)<1+j^{-1},
\end{equation}
where $\hat\rho_j=Q_j\rho_j\to\infty$, the limit is a flat plane; a contradiction.
\end{proof}

\section{Noncollapsing for mean convex mean curvature flow}

In this lecture, we discuss the noncollapsing result of Andrews. 

\begin{openproblem}[Multiplicity one question]
 Can a mean curvature flow of embedded hypersurfaces ever develop a singularity which has a higher multiplicity plane as a blowup limit?
\end{openproblem}

This is a great open problem, going back to the work of Brakke \cite{brakke}. For example, one could imagine a flow that looks more and more like two planes connected by smaller and smaller catenoidal necks.

Establishing multiplicity one is highly relevant. E.g. if the flow is weakly close to a multiplicity one plane, then the local regularity theorem (Theorem  \ref{app_thm_easy_brakke}) gives definite curvature bounds. However, if the flow is weakly close to a higher multiplicity plane, then -- as illustrated by the above example -- the curvature could be unbounded.

White proved (via clever and sophisticated arguments for Brakke flows) that blowup limits of higher multiplicity can never occur in the mean convex case \cite{white_size}, i.e. when the mean curvature is positive.
More recently, Andrews found a short quantitative argument \cite{andrews1}.

\begin{defn}[\cite{andrews1,HK1}]\label{def_andrews_static}
A closed embedded mean convex hypersurfaces $M^n\subset \RR^{n+1}$ satisfies the \emph{$\alpha$-Andrews condition},
if each $p\in M$ admits interior and exterior balls tangent at $p$ of radius $\frac{\alpha}{H(p)}$.
\end{defn}

By compactness, every closed embedded mean convex initial surface $M_0$ satisfies the Andrews condition for some $\alpha>0$. The main result of Andrews says that this is preserved under the flow.

\begin{thm}[Andrews \cite{andrews1}]\label{thm_andrews}
 If the initial surface $M_0$ satisfies the $\alpha$-Andrews condition, then so does $M_t$ for all $t\in [0,T)$.
\end{thm}

\begin{remark}
Theorem \ref{thm_andrews} immediately rules out higher multiplicity planes as potential blowup limits for mean convex mean curvature flow. It also rules out other collapsed solutions, e.g. (grim reaper) $\times$ $\R^{n-1}$.
\end{remark}

We will now describe the proof of Theorem \ref{thm_andrews}.
The first step is to express the geometric condition on the interior and exterior balls in terms of certain inequalities.
Let us first consider interior balls. For $x\in M$, the interior ball of radius $r(x)=\frac{\alpha}{H(x)}$ has the center point $c(x)=x+r(x)\nu(x)$.
The condition that this is indeed an interior ball is equivalent to the inequality
\begin{equation}\label{inequ_ymincx}
 \norm{y-c(x)}^2 \geq r(x)^2 \qquad \textrm{for all $y\in M$}.
\end{equation}
Observing $\norm{y-c(x)}^2=\norm{y-x}^2-2r(x)\langle y-x,\nu(x)\rangle +r(x)^2$ and inserting $r(x)=\frac{\alpha}{H(x)}$ the inequality (\ref{inequ_ymincx}) can be rewritten as
\begin{equation}
 \frac{2\langle y-x,\nu(x)\rangle}{\norm{y-x}^2} \leq \frac{H(x)}{\alpha} \qquad \textrm{for all $y\in M$}.
\end{equation}
Now given a mean convex flow $M_t=X(M,t)$ of closed embedded hypersurfaces, we consider the quantity
\begin{equation}\label{def_overlZ}
 Z^\ast(x,t)=\sup_{y\neq x} \frac{2\langle X(y,t)-X(x,t),\nu(x,t)\rangle}{\norm{X(y,t)-X(x,t)}^2}.
\end{equation}
Proving interior noncollapsing amounts to showing that if 
\begin{equation}\label{eq_intnoncoll}
Z^\ast(x,t)\leq \frac{H(x,t)}{\alpha}
\end{equation}
holds at $t=0$, then this holds for all $t$.
Similarly, exterior noncollapsing amounts to proving the inequality
\begin{equation}\label{eq_extnoncoll}
 Z_\ast(x,t)=\inf_{y\neq x} \frac{2\langle X(y,t)-X(x,t),\nu(x,t)\rangle}{\norm{X(y,t)-X(x,t)}^2}\geq -\frac{H(x,t)}{\alpha}.
\end{equation}

That the inequalities (\ref{eq_intnoncoll}) and (\ref{eq_extnoncoll}) are indeed preserved under mean curvature flow is a quick consequence of the following theorem.

\begin{thm}[Andrews-Langford-McCoy \cite{Andrews_Langford_McCoy}]\label{subandsupersolution}
Let $M_t$ be a mean curvature flow of closed embedded mean convex hypersurfaces, and define $Z_\ast$ and $Z^\ast$ as in (\ref{def_overlZ}) and (\ref{eq_extnoncoll}). Then
\begin{equation}\label{evol_ineq}
 \partial_t Z_\ast\geq \Lap Z_\ast+\abs{A}^2Z_\ast \qquad\qquad \partial_t Z^\ast\leq \Lap Z^\ast+\abs{A}^2Z^\ast
\end{equation}
in the viscosity sense.
\end{thm}

\begin{proof}[Proof of Theorem \ref{thm_andrews}  (using Theorem \ref{subandsupersolution})]
We start by computing
\begin{equation}
(\partial_t-\Lap) \frac{Z_\ast}{H} = \frac{(\partial_t-\Lap)\zu}{H}-\frac{\zu (\partial_t-\Lap)H}{H^2}+ 2 \langle \Grad \log H, \Grad \frac{Z_\ast}{H} \rangle.
\end{equation}
Thus, using Proposition \ref{prop_evol_eq} and Theorem \ref{subandsupersolution}, we obtain
\begin{equation}
\partial_t \frac{Z_\ast}{H}
\geq \Lap \frac{Z_\ast}{H} + 2 \langle \Grad \log H, \Grad \frac{Z_\ast}{H} \rangle.
\end{equation}
By the maximum principle, the minimum of $\frac{\zu}{H}$ is nondecreasing in time. In particular, if the inequality $\frac{\zu}{H}\geq -\frac{1}{\alpha}$ holds at $t=0$, then this inequality holds for all $t$.
Arguing similarly we obtain that 
\begin{equation}
\partial_t \frac{Z^\ast}{H}
\leq \Lap \frac{Z^\ast}{H} + 2 \langle \Grad \log H, \Grad \frac{Z^\ast}{H} \rangle,
\end{equation}
and thus that the inequality $\frac{Z^\ast}{H} \leq \frac{1}{\alpha}$ is also preserved along the flow.
\end{proof}

It remains to describe the proof of Theorem \ref{subandsupersolution}. This essentially amounts to computing various derivatives of
\begin{equation}
Z(x,y,t) = \frac{2 \langle X(y,t)-X(x,t), \nu(x,t) \rangle}{{\norm{X(y,t)-X(x,t)}}^2}.
\end{equation}
To facilitate the computation, we write $d(x,y,t)=\norm{X(y,t)-X(x,t)}$, $\omega(x,y,t) =X(y,t)-X(x,t)$, $\partial_{x^i} = \frac{\partial X(x,t)}{\partial x^i}$ and $\partial_{y^j} = \frac{\partial X(y,t)}{\partial y^j}$,  and always work in normal coordinates at $x$ and $y$, in particular we have
\begin{equation}\label{sec_fund_eqn}
\tfrac{\partial}{\partial {x^i}} \partial_{x^j} =  h_{ij}(x) \nu(x),\qquad \tfrac{\partial}{\partial {x^i}}  \nu(x)= -h_{ip}(x) \partial_{x^p}.
\end{equation}

\begin{lemma} The first derivative of $Z$ with respect to $x^i$ is given by
\begin{equation}\label{Z_first_der_x}
\frac{\partial Z}{\partial {x^i}}
= \frac{2}{d^2}\left( Z \langle \omega,\partial_{x^i} \rangle -h_{ip}(x) \langle \omega,  \partial_{x^p} \rangle  \right).
\end{equation}
\end{lemma}

\begin{proof}
Observe that $\tfrac{\partial}{\partial {x^i}} d^2= -2 \langle \omega,\partial_{x^i} \rangle$.
Using this, equation (\ref{sec_fund_eqn}), and the fact that $\langle \partial_{x^i}, \nu(x) \rangle = 0$, we compute
\begin{align*}
\frac{\partial Z}{\partial x^i}
&=  \tfrac{2}{d^2} \langle \omega, \tfrac{\partial}{\partial {x^i}}\nu(x)\rangle - \tfrac{2}{d^4} \langle \omega, \nu(x) \rangle \tfrac{\D}{\partial {x^i}} d^2\\
&= -\tfrac{2}{d^2}h_{ip}(x) \langle \omega,  \partial_{x^p} \rangle + \tfrac{2}{d^2} Z \langle \omega,\partial_{x^i} \rangle.
\end{align*}
This proves the lemma.
\end{proof}

Similarly, the first derivative of $Z$ with respect to $y^i$ is given by
\begin{equation}\label{Z_first_der_y}
\frac{\D Z}{\D {y^i}}=\frac{2}{d^2}\ip{\D_{y^i},\nu(x)-Z\omega}.
\end{equation}

\begin{exercise}[Time derivative]
Show that \begin{equation}
\D_tZ =  -\frac{2}{d^2} \left(H(x) + H(y) + \ip{\omega, \Grad H(x)}\right) + Z^2 H(x).
\end{equation}
\end{exercise}

We also need the formulas for the second spatial derivatives.

\begin{lemma} At a critical point of $Z$ with respect to $y$ we have
\begin{equation}
\frac{\partial^2 Z}{\D {x^i} \D {y^j}}= \frac{2}{d^2}(Z\delta_{ip}-h_{ip}(x))\langle \partial_{y^j}, \partial_{x^p} \rangle - \frac{2}{d^2}  \frac{\D Z}{\partial {x^i}} \langle \partial_{y^j},  \omega \rangle.
\end{equation}
\end{lemma}

\begin{proof} Differentiating (\ref{Z_first_der_y}) again, and using that we are at a critical point, we compute
\begin{align}
\frac{\partial^2 Z}{\D {x^i} \D {y^j}}
&=  \tfrac{2}{d^2} \langle \partial_{y^j}, \partial_{x^i} (\nu(x) - Z  \omega) \rangle \nonumber \\
&=  -\tfrac{2}{d^2} \langle \partial_{y^j}, h_{ip}(x) \partial_{x^p} \rangle - \tfrac{2}{d^2}\tfrac{\partial Z}{\D {x^i}}\langle \partial_{y^j},  \omega \rangle
+ \tfrac{2}{d^2}Z \langle \partial_{y^j},  \partial_{x^i} \rangle.  
\end{align}
This proves the lemma.
\end{proof}

Arguing similarly, at a critical point of $Z$ with respect to $y$ we have
\begin{equation}
\frac{\partial^2 Z}{\D {y^i} \D {y^j}}= - \frac{2}{d^2} (Z \delta_{ij}+h_{ij}(y)).
\end{equation}

\begin{exercise}[Second $x$-derivatives] Show that
\begin{multline}
\frac{\partial^2 Z}{\D {x^i} \D {x^j}}= 
Z^2h_{ij}(x)
-Zh_{ip}(x)h_{pj}(x)+\tfrac{2}{d^2}\left(h_{ij}(x)-Z\delta_{ij}\right)\\
-\tfrac{2}{d^2}\ip{\omega,\D_{x^p}}\nabla_p h_{ij}(x)+ \tfrac{2}{d^2}\ip{\omega,\D_{x^i}}\tfrac{\D Z}{\D {x^j}}
+ \tfrac{2}{d^2}  \ip{\omega,\D_{x^j}}\tfrac{\D Z}{\D {x^i}}.
\end{multline}
\end{exercise}

\begin{proof}[Proof of Theorem \ref{subandsupersolution}]
We want to show that $Z_\ast$ is a viscosity supersolution. This means, given any point $(x,t)$ and any $C^2$-function $\phi=\phi(x,t)$, with $\phi\leq Z_\ast$ in a backwards parabolic neighborhood of $(x,t)$, and equality at $(x,t)$, we have to show that
\begin{equation}\label{viscinequ}
 \partial_t \phi \geq \Lap \phi+\abs{A}^2\phi.
\end{equation}
Let $y$ be a point where the infimum in the definition of $Z_\ast$ is attained. Summing up the expressions from the above formulas, we compute
\begin{align*}
 0&\leq -\D_t(Z-\phi)+\sum_{i=1}^n\D_{x^i} \D_{x^i}(Z-\phi)+2\D_{x^i} \D_{y^i}(Z-\phi)+\D_{y^i} \D_{y^i}(Z-\phi)\nonumber\\
&=\D_t\phi-\Lap\phi-\abs{A}^2\phi+\tfrac{4}{d^2} H(x) -\tfrac{4}{d^2}h_{ip}(x)\langle \partial_{y^i}, \partial_{x^p} \rangle \nonumber\\
&\quad - \tfrac{4n}{d^2}Z+\tfrac{4}{d^2}Z\ip{\D_{y^i},\D_{x^i}}
+ \tfrac{4}{d^2}\ip{\omega,\D_{x^i}-\D_{y^i}}\tfrac{\D Z}{\D {x^i}}\nonumber\\
&=\D_t\phi-\Lap\phi-\abs{A}^2\phi\nonumber\\
&\quad+\tfrac{4}{d^2}(h_{ip}(x)-Z\delta_{ip})\left(\delta_{ip}-\ip{\D_{y^i},\D_{x^p}}+\tfrac{2}{d^2}\ip{\omega,\D_{x^p}}\ip{\omega,\D_{y^i}-\D_{x^i}}\right).
\end{align*}
By definition of $Z_\ast$, we have $h_{ip}(x)-Z\delta_{ip}\geq 0$. Moreover, it follows from an elementary geometric argument, cf. \cite[Lemma 6]{Andrews_Langford_McCoy}, that
\begin{equation}
\delta_{ip}-\ip{\D_{y^i},\D_{x^p}}+\tfrac{2}{d^2}\ip{\omega,\D_{x^p}}\ip{\omega,\D_{y^i}-\D_{x^i}}\leq 0.
\end{equation}
Putting everything together we conclude that (\ref{viscinequ}) holds. The computation for $Z^\ast$ is similar, with some signs reversed.
\end{proof}

\section{Local curvature estimate and convexity estimate}

In this lecture, we discuss the local curvature estimate (aka gradient estimate) and the convexity estimate.
The original proofs of these estimates are very involved, see White \cite{white_size,white_nature} and Huisken-Sinestrari \cite{huisken-sinestrari1,huisken-sinestrari2}.
However, our new proofs \cite{HK1} are short enough that we can discuss them in full detail in a single lecture.

Motivated by Andrews' result, we formulate the estimates for a class of flows that we call $\alpha$-Andrews flows (Definition \ref{def_alandrews}). Before stating the definition, let us recall the following three important points:

First, mean convexity and the $\alpha$-Andrews condition are both preserved under mean curvature flow, and by compactness the $\alpha$-Andrews condition is always satisfied for some constant $\alpha>0$ given any mean convex initial hypersurface (smooth, closed, embedded). 
Second, as in \cite{white_size}, it is useful to view the mean curvature flow as an equation for the closed domains $K_t$ with $M_t=\partial K_t$.
Third, if we want local estimates, then it is crucial to consider flows in any open set $U\subseteq \R^{n+1}$.

\begin{definition}[\cite{HK1}]\label{def_alandrews}
A (smooth) \emph{$\alpha$-Andrews flow} in an open set $U\subseteq \R^{n+1}$ is a (smooth) mean convex mean curvature flow $\{K_t\subseteq U\}_{t\in I}$ that satisfies the $\alpha$-Andrews condition (see Definition \ref{def_andrews_static}).
\end{definition}

We now state our first main estimate. It gives curvature control on a whole parabolic ball, from a mean 
curvature bound at a single point.

\begin{theorem}[Local curvature estimate \cite{HK1}]\label{thm-intro_local_curvature_bounds}
For all $\al>0$ there exist $\rho=\rho(\al)>0$ and $C_\ell=C_\ell(\al)<\infty$
$(\ell=0,1,2,\ldots)$ with the following property.
If $\K$ is an $\al$-Andrews flow in a parabolic ball $P(p,t,r)$ centered at a boundary point 
$p\in \D K_t$ with   $H(p,t)\leq r^{-1}$, then  
\begin{equation}\label{eqn-intro_curvature_estimate}
 \sup_{P(p,t,\rho r)}\abs{\nabla^\ell A}\leq C_\ell r^{-(\ell+1)}\, .
\end{equation}
\end{theorem}

The local curvature estimate can also be viewed as local Harnack inequality, saying that the curvatures at nearby points are comparable.

\begin{remark}
The local curvature estimate (Theorem \ref{thm-intro_local_curvature_bounds}) enables us to pass to smooth limits assuming only that we normalize the curvature at the base point; this facilitates arguments by contradiction, in particular the proof of the convexity and global curvature estimate, see below.
\end{remark}

\begin{proof}[Proof of Theorem \ref{thm-intro_local_curvature_bounds}]
Suppose the estimate doesn't hold. Then there exists a sequence $\{ \K^j\}$ of $\al$-Andrews flows defined in $P(0,0,j)$ with
$H(0,0)\leq j^{-1}$, but such that
\begin{equation}\label{eq_curv_contr}
\sup_{P(0,0,1)}|A|\geq j.
\end{equation}
We can choose coordinates such that the outward normal of $K^j_0$ at $(0,0)$ 
is $e_{n+1}$. Furthermore, by \cite[App. D]{HK1} we can assume that the sequence is admissible, i.e. that for every $R<\infty$ some time slice $K^j_{t_j}$ contains $B(0,R)$, for $j$ sufficiently large.

\begin{claim}\label{claim_halfspace} The sequence of mean curvature flows $\{\K^j\}$ converges in the
pointed Hausdorff topology to a static halfspace in $\R^{n+1}\times (-\infty,0]$, and similarly
for their complements. 
\end{claim}

\begin{proof}[Proof of Claim \ref{claim_halfspace}]
For $R<\infty$, $d>0$ let 
$
\bar B_{R,d}= \ol{B((-R+d)e_{n+1},R)}\,,
$
so $\bar B_{R,d}$ is  the closed $R$-ball tangent to
the horizontal 
hyperplane $\{ x_{n+1} = d \}$   at the point $d\,e_{n+1}$.
When $R$ is large, 
it will take time approximately $Rd$ for $\bar B_{R,d}$ to leave the upper 
halfspace $\{x_{n+1}>0\}$.  Since $0\in \D K^j_0$ for all $j$, it follows
that $\bar B_{R,d}$ cannot be contained in the interior of $K^j_t$
for any $t \in [-T,0]$, where $T \simeq Rd$.   
Thus, for large $j$ 
we can
find $d_j\leq d$
such that $\bar B_{R,{d_j}}$ has interior contact with $K^j_t$ at 
some point $q_j$, where $\langle q_j,e_{n+1}\rangle < d$, $\|q_j\|\lesssim \sqrt{Rd}$,
and moreover
$\liminf_{j\ra\infty} \langle q_j,e_{n+1}\rangle\geq 0$. 

The mean 
curvature satisfies $H(q_j,t)\leq \frac{n}{R}$.  Since $K^j_t$ satisfies 
the $\al$-Andrews condition, 
 there is a closed ball $\bar B_j$
with radius at least $\frac{\al R}{n}$ making  exterior contact with $K^j_0$
at $q_j$.
By a simple geometric calculation, this 
implies that $K^j_t$ has  height  $\lesssim \frac{d}{\al} $ in the ball
$B(0,R')$ where $R'$ is comparable  to $\sqrt{Rd}$.  As $d$ and $R$ are 
arbitrary, this implies that for any $T>0$,  and any compact subset $Y\subset\{x_{n+1}>0\}$, 
 for large $j$ the time slice
$K^j_t$ is disjoint from $Y$, for all 
$t \geq -T$. 

Finally, observe that for any $T>0$ and any
compact subset $Y\subset \{x_{n+1}<0\}$, 
the time slice
$K^j_t$ contains $Y$ for all 
$t \in[-T, 0]$, and large $j$,
because $K^j_{-T}$ contains a ball whose forward evolution under
MCF contains $Y$ at any time $t\in [-T, 0]$. This proves the claim.
\end{proof}

Finishing the proof of the theorem, by Claim \ref{claim_halfspace}, admissibility, and one-sided minimization (see below),
we get 
 for every $\eps>0$, every $t\leq 0$ and every ball $B(x,r)$ centered
on the hyperplane $\{x_{n+1}=0\}$, that 
\begin{equation}\label{eqn_densitybound}
|\D K_t^j \cap B(x,r)| \leq (1+\eps)\omega_n r^n\,,
\end{equation}
for $j$ large
enough. Hence, the local regularity theorem for the mean curvature flow (Theorem \ref{app_thm_easy_brakke}) implies $\lim\sup_{j\to\infty}\sup_{P(0,0,1)}\abs{A}=0$; this contradicts \eqref{eq_curv_contr}.
\end{proof}

\begin{exercise}[One-sided minimization]
\label{rem_one_sided_minimization}
Use Stokes' theorem to prove the following. If $\{K_{t'}\subseteq U\}_{t'\leq t}$
is a smooth family of mean convex domains 
such that $\{\D K_{t'}\}$ foliates $U\setminus\Int(K_t)$, then
\begin{equation}
 |\D K_t\cap V|\leq |\D K'\cap V|
\end{equation}
for every closed domain $K'\supseteq K_t$ which agrees with $K_t$ outside a 
compact smooth domain $V\subseteq U$. Using this, prove the density bound \eqref{eqn_densitybound}.
\end{exercise}

Our next estimate gives pinching of the curvatures towards positive.

\begin{theorem}[Convexity estimate \cite{HK1}]\label{thm-intro_convexity_estimate}
For all $\eps>0$, $\al >0$, there exists $\eta=\eta(\eps,\al)<\infty$ with the following property.
If $\K$ is an $\al $-Andrews flow in a parabolic ball  $P(p,t,\eta\, r)$ centered at a boundary point 
$p\in \D K_t$ with $H(p,t)\leq r^{-1}$, 
then
\begin{equation}
\la_1(p,t)\geq -\eps r^{-1}.
\end{equation}  
\end{theorem}

The convexity estimate (Theorem \ref{thm-intro_convexity_estimate}) says that a boundary point $(p,t)$ in an
$\al$-Andrews flow has  almost positive definite
second fundamental form, assuming only that the flow has had a chance to evolve
over a portion of spacetime which is 
large compared to $H^{-1}(p,t)$. 
In particular, ancient $\alpha$-Andrews flows $\{K_t\subset\R^{n+1}\}_{t\in (-\infty,T)}$ (e.g. blowup limits) are always convex; this is crucial for the analysis of singularities.

\begin{proof}[Proof of Theorem \ref{thm-intro_convexity_estimate}]
Fix $\al $.  The $\al$-Andrews condition implies that the assertion holds for $\eps = \frac1\al $. 
Let $\eps_0\leq \frac1\al $ be the infimum of the $\eps$'s for which it holds, 
and suppose towards a contradiction that $\eps_0>0$.

It follows that there is a sequence $\{\K^j\}$ of $\al$-Andrews flows, 
where for all $j$, $(0,0)\in \D \K^j$, $H(0,0)\leq 1$ and $\K^j$ is defined in 
$P(0,0,j)$, but ${\la_1}(0,0)\to -\eps_0$ as $j\ra \infty$.
After passing to a subsequence, 
$\{\K^j\}$  converges smoothly to a mean curvature flow $\K^\infty$ in the parabolic
ball $P(0,0,\rho)$, where $\rho=\rho(\al)$ is the quantity from Theorem
\ref{thm-intro_local_curvature_bounds}. Note that for $\K^\infty$ we have $\lambda_1(0,0)=-\eps_0$ and thus $H(0,0)=1$.

By continuity $H>\frac12$
in  $P(0,0,r)$ for some $r\in (0,\rho)$.  Furthermore we have $\frac{\la_1}{H}\geq -\eps_0$
everywhere in $P(0,0,r)$.  This is because   every $(p,t)\in \D\K^\infty\cap P(0,0,r)$
is a limit of a sequence $\{(p_j,t_j)\in  \D \K^j\}$ of boundary 
points, and  for every $\eps > \eps_0$, if 
$\eta=\eta(\eps,\al)$, then for 
large $j$, $\K^j$ is defined in $P(p_j,t_j,\eta H^{-1}(p_j,t_j))$,
which implies that 
the ratio $\frac{\la_1}{H}(p_j,t_j)$ is bounded below by $-\eps$. 
Thus, in the parabolic ball $P(0,0,r)$,  the ratio $\frac{\la_1}{H}$ attains a 
negative minimum $-\eps_0$ at $(0,0)$. Since $\lambda_1<0$ and $\lambda_{n}>0$ the Gauss curvature $K=\lambda_1\lambda_{n}$ is strictly negative.
However, by the equality case of the maximum principle for $\tfrac{\lambda_1}{H}$, the hypersurface locally splits as a product and thus this Gauss curvature must vanish; a contradiction.
\end{proof}

\section{Regularity and structure theory for weak solutions}

We start this lecture by stating our third main estimate.

\begin{theorem}[Global curvature estimate \cite{HK1}]
\label{thm-more_global_parabolic_bounds}
For all $\al>0$ and $\Lambda<\infty$, there exist 
$\eta=\eta(\al,\Lambda)<\infty$ and $C_\ell=C_\ell(\al,\Lambda)<\infty$ $(\ell =0,1,2,\ldots)$ such that if $\K$ is an $\alpha$-Andrews flow
in a parabolic ball $P(p,t,\eta r)$ centered at a boundary point $p\in \D K_t$ with $H(p,t)\leq r^{-1}$, then
\begin{equation}\label{glob_der_est}
\sup_{P(p,t,\Lambda r)}|\nabla^\ell A|\leq C_\ell r^{-(\ell+1)}\qquad (\ell=0,1,2,\ldots).
\end{equation}
\end{theorem}

\begin{remark}
Our proof of the global curvature estimate (Theorem \ref{thm-more_global_parabolic_bounds}) is based on the local curvature estimate (Theorem \ref{thm-intro_local_curvature_bounds}) and the convexity estimate (Theorem \ref{thm-intro_convexity_estimate}).
Roughly speaking, the crux of the argument is as follows:
Arguing by contradiction, we look at the supremal radius $R_0$ 
where such a curvature bound holds,  which then allows us to pass to a smooth convex
limit in the open ball $B(0,R_0)$.
Carefully examining the structure of this limit, we can then find a blowup limit whose final time-slice is a non-flat convex cone; this however can never happen under mean curvature flow. See \cite{HK1} for a detailed proof.
\end{remark}

For the rest of this lecture, we will explain how our three main estimates yield a streamlined treatment of White's regularity and structure theory for mean convex mean curvature flow \cite{white_size,white_nature,white_subsequent}. The two main theorems describe the structure and the size of the singular set in a weak flow (level set flow) of mean convex hypersurfaces.

Given any compact smooth mean convex domain $K_0\subset\R^{n+1}$, we consider the \emph{level set flow} $\{K_t\subset \R^{n+1}\}_{t\geq 0}$ starting at $K_0$ \cite{evans-spruck,CGG,Ilmanen}.
The level set flow can be defined as the maximal family of closed sets $\{K_t\}_{t\geq 0}$ starting at $K_0$ that satisfies the avoidance principle
\begin{equation*}
K_{t_0}\cap L_{t_0}=\emptyset\qquad\Rightarrow \qquad K_{t}\cap L_{t}=\emptyset
\quad \textrm{for all}\,\, t\in[t_0,t_1],
\end{equation*}
whenever $\{L_t\}_{t\in [t_0,t_1]}$ is a smooth compact mean curvature flow.
The definition is phrased in such a way that existence and uniqueness are immediate. 
Moreover, mean convexity is preserved, i.e.  $K_{t_2}\subseteq K_{t_1}$ whenever $t_2\geq t_1$.
Also,
the level set 
flow of $K_0$ coincides with
smooth mean curvature flow of $K_0$ for as long as the latter is defined.

\begin{definition}[\cite{HK1}]
\label{def_viscosity_mean_curvature}
Let $K\subseteq \R^{n+1}$ be a compact domain.   If $p\in \D K$, then the {\em viscosity mean curvature of $K$ at $p$} is
$$
H(p)=\inf\{H_{\D X}(p)\mid X\subseteq K\;\text{is a compact smooth domain,}
\;  p\in \D X\},\\
$$
where $H_{\D X}(p)$ denotes the mean curvature of $\D X$ at $p$.
By the usual convention, the infimum of the empty set is $\infty$.
\end{definition}

Based on the notion of viscosity mean curvature, we can define the $\alpha$-Andrews condition (Definition \ref{def_alandrews}) for mean convex level set flow.
Using elliptic regularization, we see that the Andrews condition is preserved also beyond the first singular time, and that the our three main estimates (Theorem \ref{thm-intro_local_curvature_bounds}, Theorem \ref{thm-intro_convexity_estimate} and Theorem \ref{thm-more_global_parabolic_bounds}) hold in the general setting of $\alpha$-Andrews (level set) flows, see \cite{HK1} for details.

Next, our main structure theorem shows that ancient $\al$-Andrews flows (e.g. blowups) are smooth and convex until they become extinct.

\begin{theorem}[Structure of ancient $\al$-Andrews flows \cite{HK1}]\label{thm_structure}
\label{thm-intro_smooth_convex_til_extinct}
Let $\K$ be an ancient $\alpha$-Andrews flow defined on $\R^N\times (-\infty,T_0)$ (typically $T_0=\infty$, but we allow $T_0<\infty$ as well),
and let $T\in (-\infty,T_0]$ be the extinction time of $\K$, 
i.e. the supremum of all $t$ with $K_t\neq\emptyset$. Then:
\begin{enumerate}
\item $\K\cap \{t<T\}$ is smooth. In fact, there exists a function $\ol{H}$ depending only on the Andrews constant $\al$ such that whenever $\tau< T-t$, then
$
H(p,t)\leq \ol{H}(\tau,d(p,K_{t+\tau})).
$
\item $\K$ has convex time slices.
\item $\K$ is either a static halfspace, or it has strictly positive mean curvature and sweeps out all space, i.e. $\bigcup_{t<T_0}K_t=\R^{n+1}$.
\end{enumerate}
Furthermore, if $\K$ is backwardly self-similar, then it is either (i) a static
halfspace or (ii) a shrinking round sphere or cylinder.
\end{theorem}

\begin{proof}[Proof of Theorem \ref{thm_structure}]

\noindent (1) Given $\tau< T-t$, we can find a boundary point $(p',t')\in\D \K$ with  $\abs{p-p'}\leq d(p,K_{t+\tau})$, $t'\in [t,t+\tau]$ and $H(p',t')\leq\frac{d(p,K_{t+\tau})}{\tau}$. Then, Theorem \ref{thm-more_global_parabolic_bounds} gives universal curvature bounds.

\noindent (2) By Theorem \ref{thm-intro_convexity_estimate}, the boundary $\D K_t$ has positive semidefinite second fundamental form for every $t$. Thus, choosing any $p\in K_{T}$, the connected component $K_t^{p} \subset K_t$ containing $p$ is convex.\\
We claim that there are no other connected components, i.e. $K_t^{p}=K_t$. Indeed, suppose for any $R<\infty$ there was another component $K'_t$ in $B(p,R)$.
Going backward in time, such a complementary component $K'_{\bar t}$ would have to stay disjoint from our principal component $K^p_{\bar t}$, and thus $K'_{\bar t}$ would have to slow down. But then the Andrews condition would 
clear out our principal component $K^p_{\bar t}$; a contradiction.

\noindent (3) If the mean curvature vanishes at some point, by Theorem \ref{thm-intro_local_curvature_bounds} and the Andrews condition the flow must be a static halfspace. If not, arguing similarly as in (2), we see that the flow sweeps out all space.

Furthermore, the argument of Huisken \cite[Sec. 5]{Huisken_local_global} shows that
any backwardly self-similar $\al$-Andrews flow with $H>0$ must be a shrinking round sphere or cylinder,
provided we can justify Huisken's partial integration for the term $\int \abs{\nabla \frac{\abs{A}^2}{H^2}}^2 e^{-\abs{x}^2/2}$
without apriori assumptions on curvature and volume. To this end, recall first that the $t=-1/2$ slice of a backwardly selfsimilar solution satisfies
\begin{equation}
 H(x)=\langle x, \nu\rangle .
\end{equation}
Together with the convexity established in part (3), this shows that the curvature grows at most linearly,
\begin{equation}
 \abs{A}\leq H\leq \abs{x},
\end{equation}
and similarly for the derivatives. Also, by the one-sided minimization property (Exercise \ref{rem_one_sided_minimization}) the volume growth is at most polynomial.
Thus, Huisken's partial integration is justified in our context.
\end{proof}

Finally, let us discuss the partial regularity theorem. Recall  that the {\em parabolic Hausdorff dimension} of a subset of spacetime, $\mathcal{S}\subset \R^{n+1}\times \R$, refers to the Hausdorff dimension
with respect to the parabolic metric on spacetime
$d((x_1,t_1),(x_2,t_2))=\max({|x_1-x_2|,|t_1-t_2|^{\frac12}})$. Note that time counts as two dimensions, e.g. $\dim(\R^{n+1}\times \R)=n+3$.

\begin{theorem}[Partial regularity \cite{HK1}]\label{thm_intro_partial_regularity}
For any $\al$-Andrews flow,
the parabolic Hausdorff dimension of the singular set is at most $n-1$.
\end{theorem}

\begin{proof}[Proof of Theorem \ref{thm_intro_partial_regularity}]
By Theorem \ref{thm_structure} and the localized version of Huisken's monotonicity formula (see Lecture 2) every tangent flow must be be either (i) a static
multiplicity one plane or (ii) a shrinking sphere or cylinder $\R^j\times \bar{D}^{n+1-j}$ with $j\leq n-1$.
By Theorem \ref{thm-intro_local_curvature_bounds} the singular set $\S\subset \D\K$ consists exactly of those boundary points where no tangent flow is a static halfspace.
Assume towards a contradiction that $\dim\mathcal{S}>n-1$. Then, blowing up at a density point we obtain a tangent flow whose singular set has parabolic Hausdorff dimension bigger than $n-1$; this contradicts the above classification of tangent flows.
\end{proof}

\begin{remark}
For refined statements on the structure of the singular set, see Cheeger-Haslhofer-Naber \cite{CHN} and Colding-Minicozzi \cite{CM_structure}.
\end{remark}

\section{Mean curvature flow with surgery}

In this final lecture, we discuss our new proof for the existence of mean curvature flow with surgery for two-convex hypersurfaces.

Roughly speaking, the idea of surgery is to continue the flow through singularities
by cutting along necks, gluing in caps, and continuing the flow of the pieces; components of known geometry and topology are discarded.
Huisken and Sinestrari successfully implemented this idea for the mean curvature flow of two-convex hypersurfaces in $\R^{n+1}$,
i.e. for hypersurfaces where the sum of the smallest two principal curvatures is positive (they have to assume in addition that $n>2$); this was the culmination of a series of long papers \cite{huisken-sinestrari1,huisken-sinestrari2,huisken-sinestrari3}.

Our new approach is comparably short and simple,\footnote{As least as simple as the extremely technical nature of the subject allows.} and also works for $n=2$ (this was also proved by Brendle-Huisken \cite{BH} using our local curvature estimate and another estimate of Brendle \cite{Brendle_inscribed}, see also \cite{HK_inscribed}). The key are new a priori estimates, in a local and flexible setting. We derive them for a class of flows that we call $(\al,\delta)$-flows.

We fix a constant $\mu\in [1,\infty)$, and a large enough constant $\Gamma<\infty$.

\begin{definition}[\cite{HK2}]\label{def_alphadelta}
An \emph{$(\alpha,\delta)$-flow} $\K$ is a collection of finitely many smooth $\al$-Andrews flows $\{K_t^i\subseteq U\}_{t\in[t_{i-1},t_{i}]}$ ($i=1,\ldots,k$; $t_0<\ldots< t_k$) in an open set $U\subseteq \R^{n+1}$,
such that
\begin{enumerate}
\item for each $i=1,\ldots,k-1$, the final time slices of some collection of disjoint strong $\delta$-necks
 are replaced by pairs of standard caps as described in Definition \ref{def_replace},
 giving a domain $K^\sharp_{t_{i}}\subseteq K^{i}_{t_{i}}=:K^-_{t_{i}}$.
\item the initial time slice of the next flow, $K^{i+1}_{t_{i}}=:K^+_{t_{i}}$, is obtained from $K^\sharp_{t_{i}}$ by discarding some connected components.
\item there exists $s_\sharp=s_\sharp(\K)>0$, which depends on $\K$, such that all necks in item (1) have radius $s\in[\mu^{-1/2}s_\sharp,\mu^{1/2} s_\sharp]$.
\end{enumerate}
\end{definition}

\begin{remark}
To avoid confusion, we emphasize that the word `some' allows for the empty set,
i.e. some of the inclusions $K_{t_i}^+\subseteq K_{t_i}^\sharp\subseteq K_{t_i}^-$ could actually be equalities. In other words, there can be some times $t_i$ where effectively only one of the steps (1) or (2) is carried out.
Also, the flow can become extinct, i.e. we allow the possibility that $K^{i+1}_{t_{i}}=\emptyset$.
\end{remark}

\begin{definition}[\cite{HK2}]
A \emph{standard cap} is a smooth convex domain $K^{\textrm{st}}\subset \R^{n+1}$ that coincides with a solid round half-cylinder of radius $1$ outside a ball of radius $10$.
\end{definition}

\begin{definition}[\cite{HK2}]
We say that an $(\al,\de)$-flow $\K=\{K_t\subseteq U\}_{t\in I}$ has a \emph{strong $\delta$-neck} with center $p$ and radius $s$ at time $t_0\in I$, if
$\{s^{-1}\cdot(K_{t_0+s^2t}-p)\}_{t\in(-1,0]}$ is $\delta$-close in $C^{\lfloor 1/\delta\rfloor}$ in $B_{1/\delta}^U\times (-1,0]$ to the evolution of a solid round cylinder $\bar{D}^{n}\times \R$ with radius $1$ at $t=0$.
\end{definition}

\begin{definition}[\cite{HK2}]\label{def_replace}
We say that the final time slice of a strong $\delta$-neck ($\delta\leq\tfrac{1}{10\Gamma}$) with center $p$ and radius $s$ is \emph{replaced by a pair of standard caps},
if the pre-surgery domain $K^-\subseteq U$ is replaced by a post-surgery domain $K^\sharp\subseteq K^-$ such that:
\begin{enumerate}
\item the modification takes places inside a ball $B=B(p,5\Gamma s)$.
 \item there are uniform curvature bounds
$$\sup_{\D K^\sharp\cap B}\abs{\nabla^\ell A}\leq C_\ell s^{-1-\ell}\qquad (\ell=0,1,2,\ldots).$$
 \item if $B\subseteq U$, then for every point $p_\sharp\in \partial K^\sharp\cap B$ with $\lambda_1(p_\sharp)< 0$, there is a point
 $p_{-}\in\partial K^{-}\cap B$ with $\frac{\lambda_1}{H}(p_{-})\leq\frac{\lambda_1}{H}(p_{\sharp})$.
 \item if $B(p,10\Gamma s)\subseteq U$, then $s^{-1}\cdot(K^\sharp-p)$ is $\delta'(\delta)$-close in $B(0,10\Gamma)$ to a pair of disjoint standard caps,
that are at distance $\Gamma$ from the origin, where $\delta'(\delta)\to 0$ as $\delta\to 0$.
\end{enumerate}
\end{definition}

\begin{theorem}[Local curvature estimate \cite{HK2}]
There exist $\bar{\delta}=\bar{\delta}(\al)>0$, $\rho=\rho(\al)>0$ and $C_\ell=C_\ell(\alpha)<\infty$ ($\ell=0,1,2,\ldots$) with the following property.
If $\K$ is an $(\al,\delta)$-flow ($\delta\leq\bar{\de}$) in a parabolic ball $P(p,t,r)$ centered at a point 
$p\in \D K_t$ with $H(p,t)\leq r^{-1}$, then
\begin{equation}\label{eqn_loccurv}
 \sup_{P(p,t,\rho r)\cap \partial \K}\abs{\nabla^\ell A}\leq C_\ell r^{-1-\ell}.
\end{equation}
\end{theorem}

\begin{theorem}[Convexity estimate  \cite{HK2}]
For all $\eps>0$, there exist $\bar{\delta}=\bar{\delta}(\alpha)>0$ and $\eta=\eta(\eps,\al)<\infty$ with the following property.
If $\K$ is an $(\al,\delta)$-flow ($\delta\leq\bar{\delta}$) defined in a parabolic ball  $P(p,t,\eta\, r)$ centered at a point 
$p\in \D K_t$ with $H(p,t)\leq r^{-1}$, then
$\la_1(p,t)\geq -\eps r^{-1}$.
\end{theorem}

\begin{theorem}[Global curvature estimate  \cite{HK2}]\label{thm_globcurvsurg}
For all $\Lambda<\infty$,
there exist $\bar{\delta}=\bar{\delta}(\alpha)>0$, $\eta=\eta(\alpha,\Lambda)<\infty$ and $C_\ell=C_\ell(\alpha,\Lambda)<\infty$ ($\ell=0,1,2,\ldots$) with the following property.
If $\K$ is an $(\al,\delta)$-flow ($\delta\leq\bar{\delta}$) in a parabolic ball  $P(p,t,\eta\, r)$ centered at a boundary point 
$p\in \D K_t$ with $H(p,t)\leq r^{-1}$, then
\begin{equation}\label{eqn_globcurv}
 \sup_{P(p,t,\Lambda r)\cap\partial \K'}\abs{\nabla^\ell A}\leq C_\ell r^{-1-\ell}.
\end{equation}
\end{theorem}

\begin{remark} The presence of surgeries makes the proof of the local curvature estimate quite delicate. Rougly speaking, the main idea is as follows.
Arguing by contradiction, we get a sequence of flows on larger and larger parabolic balls where the curvature goes to zero at the basepoint but blows up at some nearby point.
We first split off the two easy cases that there are no nearby surgeries or surgeries at macroscopic scales, which can be dealt with by applying the local curvature estimate from the previous section and the pseudolocality theorem for mean curvature flow, respectively.
The core of the proof is then to rule out surgeries at microscopic scales.
We do this as follows: At a surgery neck the value of the Huisken density is close to the value of the cylinder.
However, since the mean curvature at the basepoint goes to zero, using a halfspace convergence argument and one-sided minimization, we can show that the Huisken density is close to $1$ further back in time.
Finally, analyzing the contributions from surgeries in different regimes, we prove that the cumulative error in Huisken's monotonicity inequality due to surgeries goes to zero, and conclude that microscopic surgeries cannot occur.
In the proof of the convexity estimate and the global curvature estimate, we then follow the scheme discussed in the previous lectures.
\end{remark}

We now turn to the discussion of the existence theory.

\begin{definition}[\cite{HK2}]\label{def_initialdata}
Let $\Balpha=(\al,\beta,\gamma)\in(0,n-1)\times (0,\tfrac{1}{n-1})\times (0,\infty)$.
A smooth compact two-convex domain $K_0\subset \R^{n+1}$ is called an \emph{$\Balpha$-controlled initial condition},
if it satisfies the $\alpha$-Andrews noncollapsing-condition and the inequalities $\lambda_1+\lambda_2\geq \beta H$ and $H\leq \gamma$.
\end{definition}

\begin{remark}\label{remark_controlled}
Note that every smooth compact $2$-convex domain is a controlled initial condition for some parameters $\alpha,\beta,\gamma>0$.
\end{remark}

For us, a mean curvature flow with surgery is an $(\al,\de)$-flow with $(\al,\beta,\gamma)$-controlled initial data, subject to the following additional conditions.
First, the flow is \emph{$\beta$-uniformly two-convex}, i.e. $\lambda_1+\lambda_2\geq \beta H$.
Besides the neck-quality $\delta>0$, we have three curvature-scales $H_{\textrm{trig}}> H_{\textrm{neck}} > H_{\textrm{th}} > 1$, to which we refer as the trigger-, neck- and thick-curvature. 
The surgeries are done at times $t$ when the maximum of the mean curvature hits $H_{\textrm{trig}}$.
They are performed on a minimal disjoint collection of solid $\de$-necks of curvature $H_{\textrm{neck}}$ that separate the trigger part $\{H=H_{\textrm{trig}}\}$ from the thick part $\{H\leq H_{\textrm{th}}\}$ in $K_t^-$, and the high curvature components are discarded.
Finally, we impose the condition that surgeries are done more and more precisely, if the surgery-necks happen to be rounder and rounder.
We call our flows with surgery $(\Balpha,\de,\mathbb{H})$-flows, and the precise definition is as follows.

\begin{definition}[\cite{HK2}]\label{def_MCF_surgery}
An \emph{$(\Balpha,\de,\mathbb{H})$-flow}, $\mathbb{H}=(H_{\textrm{th}},H_{\textrm{neck}},H_{\textrm{trig}})$, is an $(\al,\de)$-flow $\{K_t\subset \R^{n+1}\}_{t\geq 0}$  with $\lambda_1+\lambda_2\geq \beta H$, and
with $\Balpha=(\al,\beta,\gamma)$-controlled initial condition $K_0\subset \R^{n+1}$ such that
\begin{enumerate}
\item $H\leq H_{\textrm{trig}}$ everywhere, and
surgery and/or discarding occurs precisely at times $t$ when $H=H_{\textrm{trig}}$ somewhere.
\item The collection of necks in item (1) of Definition \ref{def_alphadelta} is a minimal collection of solid $\de$-necks of curvature $H_{\textrm{neck}}$ which
separate the set $\{H=H_{\textrm{trig}}\}$ from $\{H\leq H_{\textrm{th}}\}$ in the domain
$K_t^-$.
\item $K_t^+$ is obtained from $K_t^\sharp$ by discarding precisely those connected components with $H>H_{\textrm{th}}$ everywhere.
In particular, of each pair of facing surgery caps precisely one is discarded. 
\item If a strong $\delta$-neck from item (2) also is a strong $\hat{\delta}$-neck for some $\hat{\delta}<\delta$, then property (4) of Definition \ref{def_replace} also holds with $\hat{\delta}$ instead of $\delta$.
\end{enumerate}
\end{definition}

\begin{remark}\label{remark_extinct}
By comparison with spheres every $(\Balpha,\de,\mathbb{H})$-flow becomes extinct after some finite time $T$, i.e. satisfies $K_t=\emptyset$ for all $t>T$.
\end{remark}

Our main existence theorem is the following.

\begin{theorem}[Existence of MCF with surgery \cite{HK2}]\label{thm_main_existence}
There are constants $\ol{\de}=\ol{\de}(\Balpha)>0$ and $\Theta(\de)=\Theta(\Balpha,\de)<\infty$ ($\delta\leq\bar{\de}$) with the following significance.
If $\de\leq\bar{\de}$ and $\mathbb{H}=(H_{\textrm{trig}},H_{\textrm{neck}},H_{\textrm{th}})$ are positive numbers with
${H_{\textrm{trig}}}/{H_{\textrm{neck}}},{H_{\textrm{neck}}}/{H_{\textrm{th}}},H_{\textrm{neck}}\geq \Theta(\de)$,
then there exists an $(\Balpha,\de,\mathbb{H})$-flow $\{K_t\}_{t\in[0,\infty)}$ for every $\Balpha$-controlled initial domain $K_0$.  
\end{theorem}

Our existence result is complemented by the following theorem.

\begin{theorem}[Canonical neighborhood theorem \cite{HK2}]\label{thm_can_nbd}
For all $\eps>0$, there exist $\ol{\de}=\ol{\de}(\Balpha)>0$, $H_{\textrm{can}}(\eps)=H_{\textrm{can}}(\Balpha,\eps)<\infty$ and $\Theta_\eps(\delta)=\Theta_\eps(\Balpha,\delta)<\infty$ ($\delta\leq\bar{\de}$) with the following significance.
If $\de\leq\ol{\de}$ and $\K$ is an $(\Balpha,\de,\mathbb{H})$-flow with ${H_{\textrm{trig}}}/{H_{\textrm{neck}}},{H_{\textrm{neck}}}/{H_{\textrm{th}}}\geq \Theta_\eps(\de)$,
then any $(p,t)\in\D \K$ with $H(p,t)\geq H_{\textrm{can}}(\eps)$ is $\eps$-close to either
(a) a $\beta$-uniformly two-convex ancient $\al$-Andrews flow,
or (b) the evolution of a standard cap preceded by the evolution of a round cylinder.
\end{theorem}

\begin{remark}
 The structure of uniformly two-convex ancient $\al$-Andrews flows and the standard solution are discussed in \cite[Sec. 3]{HK2}.
\end{remark}

\begin{corollary}\label{cor_discarded}
For $\eps>0$ small enough, for any $(\Balpha,\de,\mathbb{H})$-flow with ${H_{\textrm{trig}}}/{H_{\textrm{neck}}},{H_{\textrm{neck}}}/{H_{\textrm{th}}}\geq \Theta_\eps(\de)$ ($\de\leq\bar{\de}$) and
$H_\textrm{th}\geq H_{\textrm{can}}(\eps)$, where $\Theta_\eps(\de)$, $\bar{\de}$ and $H_{\textrm{can}}(\eps)$ are from Theorem \ref{thm_can_nbd}, all discarded components are diffeomorphic to $\bar{D}^N$ or $\bar{D}^{N-1}\times S^1$.
\end{corollary}

\begin{corollary}\label{cor_topo}
Any smooth compact $2$-convex domain in $\R^N$ is diffeomorphic to a connected sum of finitely many solid tori $\bar{D}^{N-1}\times S^1$.
\end{corollary}

\begin{remark}[Convergence to level set flow]
There exists $\bar{\delta}>0$ such that if $\K^j=\{K^j_t\}_{t\in[0,\infty)}$ is a sequence of 
$(\Balpha,\de_j,\mathbb{H}_j)$-flows ($\de_j\leq\bar{\de}$) starting at a fixed initial domain $K_0$,  with $H^j_{\textrm{th}}\to\infty$, 
then $\K^j$ Hausdorff converges in $\R^N\times [0,\infty)$ to $\K$, the level set flow of $K_0$.
\end{remark}

\begin{remark} To prove Theorem \ref{thm_can_nbd}, we have to classify the limits of sequences with degenerating $\mathbb{H}$-parameters. The existence of limits is guaranteed by the global curvature estimate (Theorem \ref{thm_globcurvsurg}).
If the limit doesn't contain surgeries, then it must be a $\beta$-uniformly two-convex ancient $\alpha$-Andrews flow, and we are done. If the limit contains a surgery, then using in particular part (2) of Definition \ref{def_MCF_surgery}, the assumption that the curvature ratios degenerate, the global curvature estimate, and the convexity estimate, we see that the limit must contain a line. It is then easy to conclude that there is in fact only one surgery, and that the limit must have the structure as claimed.
Finally, we observe that potential other connected components get cleared out.
\end{remark}

\begin{remark}
To prove Theorem \ref{thm_main_existence},
we assume towards a contradiction that we have a sequence $\K^j$ of flows with degenerating $\mathbb{H}$-parameters that can be defined only on some finite maximal time intervals $[0,T_j]$.
For $j$ large, to obtain a contradiction, we want to argue that we can perform surgery and thus continue the flow beyond $T_j$. This amounts to finding suitable collection of $\de$-necks.
To this end, we first prove that the thick and the trigger part in $K^j_{T_j}$ can be separated by a union of balls centered at boundary points with $H(p)=H_{\textrm{neck}}$ and radius comparable to $H_{\textrm{neck}}^{-1}$.
We then consider a minimal collection 
of such separating balls and prove that their centers are actually centers of strong $\hat{\delta}$-necks for any $\hat{\delta}$.
It is then easy to conclude the proof.
\end{remark}

\begin{remark}
Another exciting recent development is the generic mean curvature flow of Colding-Minicozzi \cite{CM_generic,CIM,CM_uniqueness,CM_structure}.
\end{remark}

\bibliography{lectures_mcf}

\bibliographystyle{alpha}

\vspace{10mm}
{\sc Courant Institute of Mathematical Sciences, New York University, 251 Mercer Street, New York, NY 10012, USA}

\emph{E-mail:} robert.haslhofer@cims.nyu.edu

\end{document}